\documentclass[a4paper,reqno,9pt,oneside]{amsart}
\usepackage{amsmath,geometry}
\usepackage{color}

\numberwithin{equation}{section}

\newtheorem{thm}{Theorem}[section]

\newtheorem{pro}[thm]{Proposition}
\newtheorem{lm}[thm]{Lemma}
\newtheorem{oss}[thm]{Remark}
\newtheorem{cor}[thm]{Corollary}

\begin{document}

\date{\today}

\title[A classification result for the quasi-linear Liouville equation]{A classification result for the quasi-linear Liouville equation}

\author{Pierpaolo Esposito}
\address{Pierpaolo Esposito, Dipartimento di Matematica e Fisica, Universit\`a degli Studi Roma Tre',\,Largo S. Leonardo Murialdo 1, 00146 Roma, Italy}
\email{esposito@mat.uniroma3.it}

\thanks{Partially supported by Gruppo Nazionale per l'Analisi Matematica, la Probabilit\'a  e le loro Applicazioni (GNAMPA) of the Istituto Nazionale di Alta Matematica (INdAM)}

\begin{abstract} Entire solutions of the $n-$Laplace Liouville equation in $\mathbb{R}^n$ with finite mass are completely classified.
\end{abstract}

\maketitle

\section{Introduction}
\noindent We are concerned with the following Liouville equation
\begin{equation}\label{E1}
\left\{ \begin{array}{ll}
          -\Delta_n U=e^U & \mbox{in } \mathbb{R}^n\\
         \int_{\mathbb{R}^n} e^U <+\infty &        \end{array} \right.
\end{equation}
involving the $n-$Laplace operator $\Delta_n(\cdot) =\hbox{div} (|\nabla (\cdot) |^{n-2}\nabla (\cdot) )$, $n \geq 2$. Here, a solution $U$ of \eqref{E1} stands for a function $U \in C^{1,\alpha}(\mathbb{R}^n)$ which satisfies
\begin{equation} \label{meanE1}
\int_{\mathbb{R}^n} |\nabla U|^{n-2}\langle \nabla U, \nabla \Phi\rangle=\int_{\mathbb{R}^n} e^U \Phi \qquad \forall \ \Phi \in  H=\{ \Phi \in W_0^{1,n}(\Omega): \Omega \subset \mathbb{R}^n \hbox{ bounded}\}.
\end{equation}
As wee will see, the regularity assumption on $U$ is not restrictive since a solution in $W^{1,n}_{\hbox{loc}}(\mathbb{R}^n)$ is automatically in $C^{1,\alpha}(\mathbb{R}^n)$, for some $\alpha \in (0,1)$.

\medskip \noindent Problem \eqref{E1} has the explicit solution
$$U(x)=\log\frac{c_n}{(1+|x|^{\frac{n}{n-1}})^n},\quad x \in \mathbb{R}^{n},$$
where $c_n=n (\frac{n^{2}}{n-1})^{n-1}$. Due to scaling and translation invariance, a $(n+1)-$dimensional family of explicit solutions $U_{\lambda,p}$  to \eqref{E1} is built as
\begin{equation}\label{E16bis}
U_{\lambda,p}(x)=U(\lambda(x-p))+n \log \lambda=\log\frac{c_n \lambda^n}{(1+\lambda^{\frac{n}{n-1}} |x-p|^{\frac{n}{n-1}})^n}
\end{equation}
for all $\lambda>0$ and $p \in \mathbb{R}^n$. Notice that
\begin{equation}\label{E16ter}
\int_{\mathbb{R}^n} e^{U_{\lambda,p}}=\int_{\mathbb{R}^n} e^{U}=c_n \omega_n
\end{equation}
where $\omega_n=|B_1(0)|$.
Our aim is the following classification result:
\begin{thm} \label{thm1}Let $U$ be a solution of \eqref{E1}. Then 
\begin{equation} \label{class}
U(x)=\log\frac{c_n \lambda^n}{(1+\lambda^{\frac{n}{n-1}} |x-p|^{\frac{n}{n-1}})^n},\quad x \in \mathbb{R}^{n}
\end{equation}
for some $\lambda>0$ and $p \in \mathbb{R}^n$.
\end{thm}
\noindent In a radial setting Theorem \ref{thm1} has been already proved, among other things, in \cite{KaLu}. For the semilinear case $n=2$ such a classification result is known since a long ago. The first proof goes back to J. Liouvillle \cite{Lio} who found a formula-- the so-called Liouville formula-- to represent a solution $U$ on a simply-connected domain in terms of a suitable meromorphic function. On the whole $\mathbb{R}^2$ the finite-mass condition  $\int_{\mathbb{R}^2}e^U<+\infty$ completely determines such meromorphic function.

A PDE proof has been found several years later by W. Chen and C. Li \cite{ChLi}. The fundamental point is to represent a solution $U$ of  \eqref{E1} in an integral form in terms of the fundamental solution and then deduce the precise asymptotic behavior of $U$ at infinity to start the moving plane technique. Such idea has revealed very powerful and has been also applied \cite{ChYa,Lin,Mar,WeXu,Xu1} to the higher-order version of \eqref{E1} involving the operator $(-\Delta)^{\frac{n}{2}}$. Overall, the integral equation satisfied by $U$ can be used to derive asymptotic properties of $U$ at infinity or can be directly studied through the method of moving planes/spheres. Since these methods are very well suited for integral equations, a research line has flourished about qualitative properties of integral equations, see \cite{CLO,FeXu,YYLi,Xu2,Xu3} to quote a few.

\medskip \noindent The quasi-linear case $n>2$ is more difficult. Very recently, the classification of positive $\mathcal{D}^{1,n}(\mathbb{R}^N)-$solutions to $-\Delta_n U=U^{\frac{nN}{N-n}-1}$, a PDE with critical Sobolev polynomial nonlinearity, has been achieved \cite{DMMS,Sci,Vet} for $n<N$ , see also some previous somehow related results \cite{DPR,DaRa,SeZo}.
The strategy is always based on the moving plane method and the analytical difficulty comes from the lack of comparison/maximum principles on thin strips. Moreover for $n<N$ it is not available any Kelvin type transform, a useful tool to ``gain'' decay properties on a solution. 

\medskip \noindent When $n=N$ the classical approach \cite{ChYa,ChLi,Lin,Mar,WeXu,Xu1} breaks down since an integral representation formula for a solution $U$ of \eqref{E1} is not available, due to the quasi-linear nature of $\Delta_n$. It becomes a delicate issue to determine the asymptotic behavior of $U$ at infinity and overall it is not clear how to carry out the method of moving planes/spheres. However, when $n=N$ there are some special features we aim to exploit to devise a new approach which  does not make use of moving planes/spheres, providing in two dimensions an alternative proof of the result in \cite{ChLi}. 
During the completion of this work, we have discovered that such an approach has been already used in \cite{ChKi1} for Liouville systems, where the maximum principle can possibly fail. See also \cite{KePa} for a somewhat related approach to symmetry questions in a ball. 

\medskip \noindent The case $n=N$ is usually referred to as the conformal situation, since $\Delta_n$ is invariant under Kelvin transform: $\hat U(x)=U(\frac{x}{|x|^2})$ formally satisfies
$$\Delta_n \hat U=\frac{1}{|x|^{2n}}(\Delta_n U) (\frac{x}{|x|^2}),$$
so that
$$\left\{ \begin{array}{ll}
          -\Delta_n \hat U=F(x):= \frac{e^{\hat U}}{|x|^{2n}} & \mbox{in } \mathbb{R}^n \setminus \{0\}\\
         \int_{\mathbb{R}^n} \frac{e^{\hat U}}{|x|^{2n}} <+\infty. &        \end{array} \right.$$
Equation has to be interpreted in the weak sense
$$\int_{\mathbb{R}^n} |\nabla \hat U|^{n-2}\langle \nabla \hat U, \nabla \Phi\rangle=\int_{\mathbb{R}^n} \frac{e^{\hat U}}{|x|^{2n} } \Phi \qquad \forall \ \Phi \in \hat H=\{ \Phi: \ \hat \Phi \in H\}.$$

Due to the nonlinearity of $\Delta_n$ we cannot re-absorb the factor $\frac{1}{|x|^{2n}}$ and so \eqref{E1} still does not possess any induced invariance property of Kelvin type. The behavior near an isolated singularity has been thoroughly discussed by J. Serrin \cite{Ser1,Ser2} for very general quasi-linear equations. The case $F \in L^1(\mathbb{R}^n)$ is very delicate as it represents a limiting situation where Serrin's results do not apply. Using some ideas from \cite{AgPe,BBGGPV,BoGa}, in Section 2 we first show that $U$ is bounded from above and satisfies the following weighted Sobolev estimates at infinity:
\begin{equation} \label{0825}
\int_{\mathbb{R}^n \setminus B_1(0) } \frac{|\nabla U|^q}{|x|^{2(n-q)}} <+\infty \qquad \hbox{for all }1\leq q<n.
\end{equation}
According to Remark \ref{rem0832}, estimates \eqref{0825} seem crucial to carry out in Section 3 an isoperimetric argument, which has been originally developed in \cite{ChLi} thanks to the logarithmic behavior of $U$ at infinity, to show that 
\begin{equation} \label{1541}
\int_{\mathbb{R}^n} e^U\geq c_n \omega_n, 
\end{equation}
see also \cite{Li}. Moreover, according to \cite{KaLu}, the Pohozaev identity leads to show that the equality in \eqref{1541} is valid just for solutions $U$ of the form \eqref{class}.

\medskip \noindent Thanks to \eqref{1541}, in Section 4 we can improve the previous estimates and use Serrin's type results, see \cite{Ser1,Ser2}, to show that $U$ has a logarithmic behavior at infinity along with
$$ -\Delta_n U=e^U-\gamma \delta_{\infty} \qquad \hbox{in }\mathbb{R}^n,\quad \gamma=\int_{\mathbb{R}^n} e^U.$$
Going back to an idea of Y.-Y. Li and N. Wolanski for $n=2$, the Pohozaev identity has revealed to be a fundamental tool to derive information on the mass of a singularity when $n=N$ (see for example \cite{BaTa,EsMo,MaPe,RoWe}): applied near $\infty$, it finally gives in Section 5 that $\gamma=\int_{\mathbb{R}^n} e^U=c_n \omega_n$. Notice that in Sections 2 and 4 we reproduce some estimates by emphasizing the dependence of the constants. As we will explain precisely in Remark \ref{R1}, in our argument it is crucial that all the estimates do not really depend on the structural assumption \eqref{a1}.

\medskip \noindent Problems with exponential nonlinearity on a bounded domain can exhibit non-compact solution-sequences, whose shape near a blow-up point is asymptotically described by \eqref{E1}.
A concentration-compactness principle has been established \cite{BrMe} for $n=2$  and \cite{AgPe} for $n \geq 2$. In the non-compact situation the nonlinearity concentrates at the blow-up points as a sum of Dirac measures, whose masses likely belong to $c_n \omega_n \mathbb{N}$ thanks to \eqref{E16ter}. Such a quantization for the concentration masses has been proved \cite{LiSh} for $n=2$ and extended \cite{EsMo} to $n \geq 2$ by requiring an additional boundary assumption. Very refined asymptotic properties have been later established \cite{BCLT,ChLi1,YYLi2}.
The classification result for \eqref{E1} is the starting point in all these issues, which might be now investigated also for $n \geq 2$ thanks to Theorem \ref{thm1}.

%\medskip \noindent The paper is organized as follows. In Section 1 we discuss Theorem \ref{thm1} in the %range $\int_{\mathbb{R}^n}e^U\leq c_n \omega_n$. Section 2 is devoted to some estimates which will be %used in Section $3$  to complete the proof of Theorem \ref{thm1}.

%\medskip\noindent
%{\bf Acknowledgments:}

\section{Some estimates}
\noindent Let $\Omega \subset \mathbb{R}^n$ be a bounded domain and ${\bf a}:\Omega \times \mathbb{R}^n \to \mathbb{R}^n$ be a Carath\'eodory function so that
\begin{eqnarray}
|{\bf a}(x,p)|\leq c(a(x)+|p|^{n-1}) && \forall p \in \mathbb{R}^n, \ a.e.\ x \in \Omega \label{a1}\\
\langle {\bf a}(x,p)-{\bf a}(x,q),p-q \rangle   \geq d |p-q|^n && \forall p,q \in \mathbb{R}^n,\ a.e. \ x \in \Omega \label{a2}
\end{eqnarray}
for some $c,d>0$ and $a \in L^{\frac{n}{n-1}}(\Omega)$. Given $f\in L^1(\Omega)$, let $u \in W^{1,n}(\Omega)$ be a weak solution of
\begin{equation} \label{1420}
-\hbox{div}\ {\bf a}(x,\nabla u)=f \qquad \mbox{in } \Omega.
\end{equation}
Thanks to \eqref{a1} equation \eqref{1420} is interpreted in the following sense:
\begin{equation} \label{1725}
\int_{\Omega} \langle {\bf a}(x,\nabla u), \nabla \phi \rangle=\int_\Omega f \phi \qquad \forall \phi \in W^{1,n}_0(\Omega)\cap L^\infty(\Omega).
\end{equation}
Since $u \in W^{1,n}(\Omega)$ let us consider the weak solution $h\in W^{1,n}(\Omega)$ of
\begin{equation} \label{1419bis}
\left\{ \begin{array}{ll} \hbox{div}\ {\bf a}(x,\nabla h)=0 &\hbox{in }\Omega\\
h=u& \hbox{on }\partial \Omega. \end{array} \right.
\end{equation}
Introduce the truncature operator $T_k$, $k>0$, as
\begin{equation} \label{1426}
T_k(u)=\left\{ \begin{array}{cl}  u &\hbox{if }|u|\leq k\\ k \frac{u}{|u|}  &\hbox{if }|u|>k. \end{array} \right.
\end{equation}
\noindent According to \cite{AgPe,BBGGPV,BoGa} we have the following estimates.
\begin{pro}  \label{T4}
Let $f\in L^{1}(\Omega)$ and assume \eqref{a1}-\eqref{a2}. Let $u$ be a weak solution of  \eqref{1420} in the sense \eqref{1725}, and set 
$$\Lambda_q =(\frac{S_q^{\frac{n}{q}}   d}{\|f\|_1})^{\frac{1}{n-1}}$$ 
where $S_q$ is the Sobolev constant for the embedding $\mathcal D^{1,q}(\mathbb{R}^n) \hookrightarrow L^{\frac{nq}{n-q}}(\mathbb{R}^n)$, $1\leq q<n$. Then,  for every $0<\lambda< \Lambda_1$ there hold
\begin{eqnarray}
\label{E11}
\int_{\Omega}e^{\lambda |u-h|} \leq \frac{ |\Omega|}{ 1-\lambda \Lambda_1^{-1}} ,\qquad \int_\Omega |\nabla (u-h)|^q \leq  \frac{2 S_q }{\Lambda_q ^{\frac{q(n-1)}{n}}} 
\left( 1+\frac{2^{\frac{n}{q(n-1)}}}{(n-1)^{\frac{1}{n-1}}\Lambda_q } \right)^{\frac{q}{n}}|\Omega|^{\frac{n-q}{n}}.
\end{eqnarray}
\end{pro}
\begin{proof} 
Fix $k\geq 0 $, $a> 0$. Since $T_{k+a}(u-h)-T_k(u-h) \in W^{1,n}_0(\Omega)\cap L^\infty (\Omega)$, by \eqref{1725}-\eqref{1419bis} we get that
\begin{equation} \label{1658}
\int_\Omega \langle {\bf a}(x,\nabla u)-{\bf a}(x,\nabla h), \nabla\ [T_{k+a}(u-h)-T_k (u-h)] \rangle= \int_\Omega f [T_{k+a}(u-h)-T_k (u-h)],
\end{equation}
yielding to
\begin{equation} \label{1438}
\frac{1}{a}\int_{\{k< |u-h| \leq k+a\} } |\nabla (u-h)|^n \leq \frac{ \|f\|_1}{d}
\end{equation}
in view of \eqref{a2}.  By \eqref{1438} and the following Lemma we deduce the validity of \eqref{E11} and the proof of Proposition \ref{T4} is complete.
\end{proof}
\begin{lm} \label{lmACP}
Let $w$ be a measurable function with $T_k(w) \in W^{1,n}_0(\Omega)$ so that for all $k\geq 0$, $a > 0$
\begin{equation} \label{1846}
 \frac{1}{a}\int_{\{k< |w| \leq k+a\} } |\nabla w|^n \leq C_0
\end{equation}
for some $C_0>0$. Then there hold
\begin{equation} \label{1847}
\int_{\Omega}e^{\lambda |w|} \leq \frac{ |\Omega|}{ 1-\lambda \Lambda^{-1}}, \qquad \int_\Omega |\nabla w|^q \leq  2 C_0^{\frac{q}{n}}\left( 1+(\frac{2^{\frac{n}{q}} C_0}{(n-1)S_q^{\frac{n}{q}}})^{\frac{1}{n-1}} \right)^{\frac{q}{n}}|\Omega|^{\frac{n-q}{n}}
\end{equation}
for every $0<\lambda<\Lambda=(\frac{S_1^n}{C_0})^{\frac{1}{n-1}}$ and $1\leq q<n$, where $k_0$ is given in \eqref{k0}.
\end{lm}
\begin{proof}
Let $\Phi(k)= |\{x \in \Omega:|w(x)|>k\}|$ be the distribution function of $|w|$. We have that
\begin{eqnarray*}
\Phi(k+a)^{\frac{n-1}{n}} &\leq& \frac{1}{a} \left(\int_\Omega |T_{k+a}(w)-T_k(w)|^{\frac{n}{n-1}}\right)^{\frac{n-1}{n}}\leq \frac{1}{a S_1} 
\int_\Omega |\nabla T_{k+a}(w)-\nabla T_k(w)|\\
& =& \frac{1}{a S_1}  \int_{\{k< |w| \leq k+a\} } |\nabla w|
\end{eqnarray*}
where $S_1$ is the Sobolev constant of the embedding $\mathcal D^{1,1}(\mathbb{R}^n) \hookrightarrow L^{\frac{n}{n-1}}(\mathbb{R}^n)$. By the H\"older's inequality and \eqref{1846} we deduce that
$$\Phi (k+a) \leq \frac{\Phi(k)-\Phi (k+a)}{a \Lambda}$$
and, as $a \to 0^+$,
\begin{equation} \label{0946}
\Phi (k) \leq - \frac{1}{\Lambda}  \Phi'(k)
\end{equation}
for a.e. $k>0$. Since $\Phi$ is a monotone decreasing function, an integration of \eqref{0946} 
$$\ln \frac{\Phi(k)}{\Phi(0)}\leq \int_0^k \frac{\Phi'}{\Phi} ds \leq-\Lambda k$$
provides that 
$$\Phi(k)\leq |\Omega| e^{- \Lambda k}$$
for all $k>0$, and then 
\begin{eqnarray*} \int_{\Omega}e^{\lambda |w|}
&=& |\Omega|+\lambda \int_\Omega dx \int_0^{|w(x)|} e^{\lambda k} dk= |\Omega|+\lambda \int_0^{\infty} e^{\lambda k}  \Phi (k) dk\\
&\leq&  |\Omega|+\lambda |\Omega|   \int_0^{\infty} e^{-(\Lambda-\lambda) k} dk= \frac{ |\Omega|}{ 1-\lambda \Lambda^{-1}} 
\end{eqnarray*}
for all $0<\lambda<\Lambda$. Given $k_0 \in \mathbb{N}$ introduce the sets 
$$\Omega_{k_0}=\{x \in \Omega:\ |w(x)|\leq k_0\},\qquad \Omega_k=\{x \in \Omega:\ k-1<|w(x)|\leq k\} \ (k>k_0),$$
and by the H\"older's inequality write for $1\leq q <n$
$$\int_{\Omega_{k_0}}|\nabla w|^q \leq (C_0 k_0)^{\frac{q}{n}}|\Omega|^{\frac{n-q}{n}},\quad 
\int_{\Omega_k }|\nabla w|^q \leq C_0 ^{\frac{q}{n}}|\Omega_k|^{\frac{n-q}{n}}\leq \frac{C_0^{\frac{q}{n}}}{(k-1)^q} \left(  \int_{\Omega_k} |w|^{\frac{nq}{n-q}}  \right)^{\frac{n-q}{n}}$$
thanks to \eqref{1846}. For $N\in \mathbb{N}$ let us sum up to get by the H\"older's inequality 
\begin{eqnarray}
\int_\Omega |\nabla T_{k_0+N} (w)|^q&=&\sum_{k=k_0}^{k_0+N} \int_{\Omega_k}|\nabla w|^q \leq
(C_0 k_0)^{\frac{q}{n}} |\Omega|^{\frac{n-q}{n}}+C_0^{\frac{q}{n}} \left(\sum_{k=k_0+1}^{k_0+N}  \frac{1}{(k-1)^n} \right)^{\frac{q}{n}}
\left(\sum_{k=k_0+1}^{k_0+N}  \int_{\Omega_k} |w|^{\frac{nq}{n-q}}  \right)^{\frac{n-q}{n}} \nonumber \\
&\leq&
(C_0 k_0)^{\frac{q}{n}}|\Omega|^{\frac{n-q}{n}}+C_0^{\frac{q}{n}} \left(\sum_{k=k_0+1}^{k_0+N}  \frac{1}{(k-1)^n} \right)^{\frac{q}{n}}
\left(\int_{\Omega} |T_{k_0+N} (w)|^{\frac{nq}{n-q}}  \right)^{\frac{n-q}{n}}.\label{0957}
\end{eqnarray}
Letting
\begin{equation} \label{k1}
k_0 = 1+(\frac{2^{\frac{n}{q}} C_0}{(n-1)S_q^{\frac{n}{q}}})^{\frac{1}{n-1}},
\end{equation}
we have that
\begin{equation} \label{k0}
\sum_{k\geq k_0}  \frac{1}{k^n} \leq \int_{k_0-1}^\infty \frac{dt}{t^n}=\frac{(k_0-1)^{-(n-1)}}{n-1}= \frac{1}{C_0} (\frac{S_q}{2})^{\frac{n}{q}}.
\end{equation}
By using the Sobolev embedding $\mathcal D^{1,q}(\mathbb{R}^n) \hookrightarrow L^{\frac{nq}{n-q}}(\mathbb{R}^n)$ on the L.H.S. of \eqref{0957} and by \eqref{k0} we deduce that
\begin{eqnarray*}
S_q \left(\int_\Omega |T_{k_0+N} (w)|^{\frac{nq}{n-q}} \right)^{\frac{n-q}{n}} \leq
2(C_0 k_0)^{\frac{q}{n}}|\Omega|^{\frac{n-q}{n}},
\end{eqnarray*}
which inserted into \eqref{0957} gives in turn
$$\int_\Omega |\nabla T_{k_0+N} (w)|^q \leq 2(C_0 k_0)^{\frac{q}{n}}|\Omega|^{\frac{n-q}{n}}.$$
Letting $N\to +\infty$ we finally deduce that
$$\int_\Omega |\nabla w|^q \leq 2(C_0 k_0)^{\frac{q}{n}}|\Omega|^{\frac{n-q}{n}}=2 C_0^{\frac{q}{n}}\left( 1+(\frac{2^{\frac{n}{q}} C_0}{(n-1)S_q^{\frac{n}{q}}})^{\frac{1}{n-1}} \right)^{\frac{q}{n}}|\Omega|^{\frac{n-q}{n}}$$
in view of \eqref{k1} and the proof is complete.
\end{proof}

\noindent As a first by-product of Proposition \ref{T4} we have that
\begin{thm} \label{1745} Let $U \in W^{1,n}_{\hbox{loc}}(\mathbb{R}^n)$ be a weak solution of \eqref{E1}. Then $\displaystyle \sup_{\mathbb{R}^n} U<+\infty$ and $U\in C^{1,\alpha}(\mathbb{R}^n)$, $\alpha \in (0,1)$. 
\end{thm}
\begin{proof} Assume that for $0<\epsilon \leq 1$
\begin{equation}\label{1840}
\int_{B_\epsilon(x)} e^U \leq \frac{S_1^n d}{3^{n-1}}.
\end{equation}
Thanks to Proposition \ref{T4} by \eqref{1840} we deduce that
\begin{equation}\label{1539}
\int_{B_\epsilon(x)} e^{2|U-H|}\leq 3 \omega_n,
\end{equation}
where $H$ is a $n-$harmonic function in $B_\epsilon(x)$ with $H=U$ on $\partial B_\epsilon(x)$. Since $H\leq U$ on $B_\epsilon(x)$ by the comparison principle, we have that
\begin{equation} \label{1557}
\int_{B_\epsilon(x)} H_+^n \leq \int_{B_\epsilon(x)} U_+^n  \leq n! \int_{\mathbb{R}^n} e^U
\end{equation}
where $u_+$ denotes the positive part of $u$. Since Theorem $2$ in \cite{Ser1} is easily seen to be valid for $H^+$ too (simply by replacing $|H|$ with $H^+$ in the proof), by \eqref{1557} we have that 
\begin{equation} \label{15400}
\sup_{B_{\frac{\epsilon}{2}}(x)} H_+ \leq C_0(\epsilon)
\end{equation}
for some $C_0(\epsilon)>0$ independent on $x$. By \eqref{1539} and \eqref{15400} we deduce that
\begin{equation} \label{1556}
\int_{B_{\frac{\epsilon}{2}}(x)} e^{2U}=
\int_{B_{\frac{\epsilon}{2}}(x)} e^{2|U-H|} e^{2H}\leq
3 e^{2C_0(\epsilon)} \omega_n.
\end{equation}
Still thanks to the elliptic estimates in \cite{Ser1} on $U^+$, by \eqref{1557} and \eqref{1556} we have that
\begin{equation} \label{1834}
\sup_{B_{\frac{\epsilon}{4}}(x)} U_+ \leq C_1(\epsilon)
\end{equation}
for some $C_1(\epsilon)>0$ independent on $x$. To complete the proof, we argue as follows. Since $\int_{\mathbb{R}^n} e^U<+\infty$ we can find $R>0$ so that 
\begin{equation} \label{1517}
\int_{\mathbb{R}^n \setminus B_R(0)} e^U \leq \frac{S_1^n d}{3^{n-1}}.
\end{equation}
Given $|x|> R+1$, by \eqref{1517} we have the validity of \eqref{1840} with $\epsilon=1$. For all $|x|\leq R+1$ we can find $\epsilon_x>0$ small so that \eqref{1840} holds. By the compactness of the set $\{|x|\leq R+1\}$ we can find points $x_1,\dots, x_L$ so that
\begin{equation} \label{1853}
\{|x|\leq R+1\} \subset  \bigcup_{i=1}^L  B_{\frac{\epsilon_{x_i}}{4}}(x_i).
\end{equation}
Therefore, by \eqref{1834} we deduce that
$$\sup_{\mathbb{R}^n} U \leq \max\{C_1(1),C_1(\epsilon_{x_1}),\dots, C_1(\epsilon_{x_L})  \}<+\infty$$
in view of \eqref{1853}. Since $e^U \in L^\infty (\mathbb{R}^n)$ and $U \in L^n_{\hbox{loc}}(\mathbb{R}^n)$, we can use the elliptic estimates in
\cite{Dib,Ser1,Tol} to show that $U \in C^{1,\alpha}(\mathbb{R}^n)$, for some $\alpha \in (0,1)$.
\end{proof}

\noindent We aim now to establish some bounds on $U$ at infinity. Let us recall that the Kelvin transform $\hat U(x)=U(\frac{x}{|x|^2})$ of $U$ satisfies
\begin{equation}\label{E1equiv}
\left\{\begin{array}{ll}
-\Delta_n \hat U=\frac{e^{\hat U}}{|x|^{2n}}&  \hbox{in }\mathbb{R}^n \setminus \{0\}\\
\int_{\mathbb{R}^n} \frac{e^{\hat U}}{|x|^{2n}} <+\infty,& \end{array}\right.
\end{equation}
where the equation is meant in the weak sense
\begin{equation} \label{meanE1equiv}
\int_{\mathbb{R}^n} |\nabla \hat U|^{n-2} \langle \nabla \hat U,\nabla \Phi \rangle=\int_{\mathbb{R}^n} \frac{e^{\hat U}}{|x|^{2n}}  \Phi \qquad \forall \Phi \in \hat H=\{ \Phi :\  \hat \Phi \in H\}
\end{equation} 
with $H$ given in \eqref{meanE1}. By Theorem \ref{1745} we know that $\hat U \in C^{1,\alpha}(\mathbb{R}^n\setminus \{0\})$. Here and in the sequel, $\alpha \in (0,1)$ will denote an H\"older exponent which can varies from line to line.

\medskip \noindent In order to understand the behavior of $\hat U$ at $0$, we fix $r>0$ small and, for all $0<\epsilon<r$, let $H_\epsilon \in W^{1,n}(A_\epsilon) $ satisfy
\begin{equation} \label{eqHepsilon}
\left\{ \begin{array}{ll} \Delta_n H_\epsilon=0 &\hbox{in }A_\epsilon:=B_r(0) \setminus B_\epsilon(0)\\
H_\epsilon=\hat U&\hbox{on }\partial A_\epsilon. \end{array} \right.
\end{equation}
Regularity issues for quasi-linear PDEs involving $\Delta_n$ are well established since the works of DiBenedetto, Evans, Lewis, Serrin, Tolksdorf, Uhlenbeck, Uraltseva. For example, local H\"older estimates on $H_\epsilon$ can be found in \cite{Ser1} and then it follows by \cite{Dib,Tol} that $H_\epsilon \in C^{1,\alpha}(A_\epsilon)$. Thanks to \cite{Lie} such regularity can be pushed up to the boundary to deduce that $H_\epsilon \in C^{1,\alpha}(\overline{A_\epsilon})$. By \eqref{eqHepsilon} the function $U_\epsilon=\hat U-H_\epsilon \in C^{1,\alpha}(\overline{A_\epsilon})$ satisfies
\begin{equation} \label{eqUepsilon}
\left\{ \begin{array}{ll} \Delta_n (\hat U-U_\epsilon)=0 &\hbox{in }A_\epsilon\\
U_\epsilon=0&\hbox{on }\partial A_\epsilon. \end{array} \right.
\end{equation}
We aim to derive estimates on $H_\epsilon$ and $U_\epsilon$ on the whole $A_\epsilon$ by using Proposition \ref{T4} with
\begin{equation} \label{2022}
{\bf a}(x,p)=|\nabla \hat U(x)|^{n-2}\nabla \hat U(x)-|\nabla \hat U(x)-p|^{n-2}(\nabla \hat U(x)-p).
\end{equation}
\begin{oss} \label{R1}
Let us notice that ${\bf a}(x,p)$ in \eqref{2022} satisfies \eqref{a1} with $a=|\nabla \hat U|^{n-1}$. Since $\hat U$ is expected to be singular at $0$, it is likely true that $\|a\|_{\frac{n}{n-1},A_\epsilon}\to +\infty$ as $\epsilon \to 0$. In order to get uniform estimates in $\epsilon$, it is crucial that the estimates in Propositions \ref{T4} do not depend on $\|a\|_{L^{\frac{n}{n-1}}(\Omega)}$. Assumption \eqref{a1} is just necessary to make meaningful the notion of $W^{1,n}-$weak solution for \eqref{1420}. The same remark is in order for Proposition \ref{Serrinest}, when we will use it in Section 4 to show the logarithmic behavior of $\hat U$ at $0$. 
\end{oss}
\noindent As a second by-product of Proposition \ref{T4} we have that
\begin{thm} \label{aiuto}
There holds
\begin{equation} \label{pro}
\hat U \in W^{1,q}_{\hbox{loc}}(\mathbb{R}^n) 
\end{equation}
for all $1\leq q<n$. \end{thm}
\begin{proof}
Since \eqref{E1equiv} does hold in $A_\epsilon$, \eqref{eqUepsilon} can be re-written as
\begin{equation} \label{eqUepsilonequiv}
\left\{ \begin{array}{ll} \Delta_n (\hat U-U_\epsilon)-\Delta_n \hat U=\frac{e^{\hat U}}{|x|^{2n}} &\hbox{in }A_\epsilon\\
U_\epsilon=0&\hbox{on }\partial A_\epsilon. \end{array} \right.
\end{equation}
Since 
\begin{equation} \label{1045}
d= \inf_{v\not= w} \frac{\langle |v|^{n-2}v-|w|^{n-2}w,v-w \rangle}{|v-w|^n}>0,
\end{equation}
we can apply Proposition \ref{T4} to ${\bf a}(x,p)$ in \eqref{2022}. Since $|A_\epsilon|\leq \omega_n r^n$ and ${\bf a}(x,0)=0$, we deduce that
\begin{equation} \label{17255}
\int_{A_\epsilon} |\nabla U_\epsilon|^q +\int_{A_\epsilon} e^{p U_\epsilon} \leq C
\end{equation}
for all $1\leq q<n$ and all $p\geq 1$ if $r$ is sufficiently small, where $C$ is uniform in $\epsilon$. Notice that
$$\int_{B_r(0)} \frac{e^{\hat U}}{|x|^{2n}}=\int_{\mathbb{R}^n \setminus B_{\frac{1}{r}}(0)} e^U  \to 0$$
as $r \to 0$. By the Sobolev embedding $\mathcal D^{1,\frac{n}{2}}(\mathbb{R}^n) \hookrightarrow L^n(\mathbb{R}^n)$ estimate \eqref{17255} yields that 
\begin{equation} \label{1742}
\int_{A_\epsilon} |U_\epsilon|^n \leq C
\end{equation}
for some $C$ uniform in $\epsilon$. Since $H_\epsilon=\hat U-U_\epsilon$ with $\hat U \in C^{1,\alpha}(\mathbb{R}^n\setminus \{0\})$, by \eqref{1742} we deduce that
$$\|H_\epsilon\|_{L^n(A)} \leq C(A) \qquad  \forall \ A \subset \subset \overline{B_r(0)} \setminus \{0\}$$
for all $\epsilon$ sufficently small. Arguing as before, by \cite{Dib,Lie,Ser1,Tol} it follows that
$$\|H_\epsilon\|_{C^{1,\alpha}(A)} \leq C(A) \qquad \forall \ A \subset \subset \overline{B_r(0)} \setminus \{0\}$$
for $\epsilon$ small. By the Ascoli-Arzel\'a's Theorem and a diagonal process, we can find a sequence $\epsilon \to 0$ so that $H_\epsilon \to H_0$ in $C^1_{\hbox{loc}} (\overline{B_r(0)}\setminus \{0\})$, where $H_0$ satisfies
$$\left\{\begin{array}{ll}
\Delta_n H_0=0 & \hbox{in } B_r(0) \setminus \{0\}\\
H_0=\hat U &\hbox{on }\partial B_r(0). \end{array} \right. $$
Since $H_\epsilon \leq \hat U$ in $A_\epsilon$ by the comparison principle, we have that $U_\epsilon \to U_0:=\hat U-H_0$ in $C^1_{\hbox{loc}} (\overline{B_r(0)}\setminus \{0\})$, where $U_0$ satisfies
$$ U_0\geq 0 \hbox{ in }B_r(0) \setminus \{0\},\qquad \partial_\nu U_0 \leq 0 \hbox{ on }\partial B_r(0).$$
Moreover, by \eqref{17255} we get that 
\begin{equation} \label{pro1}
U_0 \in W^{1,q}_0(B_r(0)) ,\qquad e^{U_0} \in L^p(B_r(0))
\end{equation}
for all $1\leq q<n$ and all $p\geq 1$ if $r$ is sufficiently small.

\medskip \noindent Since $H_0$ is a continuous $n-$harmonic function in $B_r(0) \setminus \{0\}$ with  
$$H_0 \leq \sup_{\mathbb{R}^n \setminus \{0\}} \hat U=\sup_{\mathbb{R}^n} U<\infty$$
in view of Theorem \ref{1745}, we can apply the result in \cite{Ser1} about isolated singularities: either $H_0$ has a removable singularity at $0$ or 
$$\frac{1}{C} \leq \frac{H_0(x)}{\ln |x|}\leq C$$
near $0$ for some $C>1$. According to \cite{Ser2}, in both situations we have that 
\begin{equation} \label{pro2}
H_0 \in W^{1,q}(B_r(0)) 
\end{equation}
for all $1\leq q<n$. The combination of \eqref{pro1} and \eqref{pro2} establishes the validity of \eqref{pro} for $\hat U=U_0+H_0$. 
\end{proof}

\noindent In terms of $U$, Theorem \ref{aiuto} simply gives that
\begin{cor} \label{aiutocor}
There holds
\begin{equation} \label{1712}
\int_{\mathbb{R}^n \setminus B_1(0)} \frac{|\nabla U|^q}{|x|^{2(n-q)}} <+\infty
\end{equation}
for all $1\leq q <n$.
\end{cor}
\begin{proof}
Since 
$$\Big| \hbox{det }D \frac{x}{|x|^2} \Big|=\frac{1}{|x|^{2n}}$$
and
$$|\nabla \hat U|(x)=\frac{1}{|x|^2}|\nabla U|(\frac{x}{|x|^2}),$$
we have that
$$\int_{B_r(0)}|\nabla \hat U|^q=\int_{\mathbb{R}^n \setminus B_{\frac{1}{r}}(0)}
\frac{|\nabla U|^q}{|x|^{2(n-q)}}.$$
By Theorems \ref{1745} and \ref{aiuto} we then deduce that
$$\int_{\mathbb{R}^n \setminus B_1(0)} \frac{|\nabla U|^q}{|x|^{2(n-q)}} <+\infty$$
for all $1\leq q <n$, as desired.
\end{proof}

\section{An isoperimetric argument}
\medskip \noindent The aim is to classify all the solutions $U$ of \eqref{E1} with small ``mass''. The following isoperimetric approach leads to:
\begin{thm} \label{iso0944}
Let $U$ be a solution of \eqref{E1} with $\int_{\mathbb{R}^n} e^U\leq c_n \omega_n$. Then $U$ is given by \eqref{E16bis}. 
\end{thm}
\begin{proof} Since $U \in C^{1,\alpha}(\mathbb{R}^n)$, we can use Theorem 3.1 in \cite{Sci1} to get that $Z_k=\{x\in B_k(0): \nabla U(x)=0\}$ is a null set for all $k \in \mathbb{N}$. By the Lipschitz continuity of $U$ on $B_k(0)$, we deduce that 
$$\{t \in \mathbb{R}: \, \exists \ x \in \mathbb{R}^n \hbox{ s.t. }U(x)=t,\ \nabla U(x)=0\}=\bigcup_{k \in \mathbb{N}} U(Z_k)$$
is a null set in $\mathbb{R}$. Therefore $\Omega_t=\{U>t\}$ is a smooth set for a.e. $t \leq t_0$, $t_0=\displaystyle \sup_{\mathbb{R}^n}U$, and has bounded Lebesgue measure in view of $\displaystyle \int_{\mathbb{R}^n}e^U<+\infty$.

\medskip \noindent Let $t \leq t_0$ and $r>0$. Given $\delta,\eta>0$, let us define the following functions:
$$\chi_\delta(s)=\left\{ \begin{array}{ll} 0 &\hbox{if }s\leq t \\ \frac{s-t}{\delta}&\hbox{if }t\leq s \leq t+\delta \\ 1 &\hbox{if }s\geq t+\delta \end{array} \right. $$
and
$$\chi_\eta(x)=\left\{ \begin{array}{ll} 1 &\hbox{if }x \in B_r(0) \\ \frac{r+\eta-|x|}{\eta}&\hbox{if } x \in B_{r+\eta}(0)\setminus B_r(0) \\ 0 &\hbox{if }x \notin B_{r+\eta}. 
\end{array} \right. $$
We can use $\chi_\delta(U) \chi_\eta(x)$ as a test function in \eqref{meanE1} to get
\begin{equation} \label{1753}
\int_{\mathbb{R}^n} e^U \chi_\delta (U) \chi_\eta(x)=\frac{1}{\delta}\int_{\Omega_t \setminus \Omega_{t+\delta}} \chi_\eta |\nabla U|^n-\frac{1}{\eta}
\int_{B_{r+\eta}(0)\setminus B_r(0)} \chi_\delta (U) |\nabla U|^{n-2} \langle \nabla U, \frac{x}{|x|} \rangle.
\end{equation}
By the Lebesgue's monotone convergence theorem for the first term in the R.H.S. of \eqref{1753} we have that
$$\frac{1}{\delta} \int_{\Omega_t  \setminus \Omega_{t+\delta}} \chi_\eta |\nabla U|^n \to \frac{1}{\delta} \int_{(\Omega_t  \setminus \Omega_{t+\delta})\cap B_r(0)} |\nabla U|^n$$
as $\eta \to 0$. Since by the co-area formula we can write  
$$\int_{(\Omega_t  \setminus \Omega_{t+\delta})\cap B_r(0)}  |\nabla U|^n=\int_t^{t+\delta}  ds \int_{\partial \Omega_s \cap B_r(0)}  |\nabla U|^{n-1} d\sigma,$$
it results that the function $t \to \int_{\partial \Omega_t \cap B_r(0)} |\nabla U|^{n-1} d\sigma$ is in $L^1_{\hbox{loc}}(\mathbb{R})$, and as $\delta \to 0$ by the Lebesgue's differentiation Theorem we conclude that for a.e. $t \leq t_0$
\begin{equation} \label{150111}
\frac{1}{\delta}\int_{\Omega_t \setminus \Omega_{t+\delta}} \chi_\eta |\nabla U|^n \to \int_{\partial \Omega_t \cap B_r(0)} |\nabla U|^{n-1}d \sigma
\end{equation}
as $\eta \to 0$ and $\delta \to 0$. The second term in the R.H.S. of \eqref{1753} writes in radial coordinates as
$$\frac{1}{\eta} \int_{B_{r+\eta}(0)\setminus B_r(0)} \chi_\delta (U) |\nabla U|^{n-2} \langle \nabla U, \frac{x}{|x|} \rangle=\frac{1}{\eta}\int_r^{r+\eta}  ds \int_{\partial B_s(0)} \chi_\delta (U) |\nabla U|^{n-2} \langle \nabla U, \frac{x}{|x|} \rangle d \sigma,$$
and by the fundamental Theorem of calculus we get that for all $r>0$
$$\frac{1}{\eta}
\int_{B_{r+\eta}(0)\setminus B_r(0)} \chi_\delta (U) |\nabla U|^{n-2} \langle \nabla U, \frac{x}{|x|} \rangle
\to \int_{\partial B_r(0)} \chi_\delta (U) |\nabla U|^{n-2} \langle \nabla U, \frac{x}{|x|} \rangle d \sigma
$$
as $\eta \to 0$. By the Lebesgue's monotone convergence theorem we deduce that for all $r>0$
\begin{equation} \label{15012}
\frac{1}{\eta}
\int_{B_{r+\eta}(0)\setminus B_r(0)} \chi_\delta (U) |\nabla U|^{n-2} \langle \nabla U, \frac{x}{|x|} \rangle
\to \int_{\Omega_t \cap \partial B_r(0)} |\nabla U|^{n-2} \langle \nabla U, \frac{x}{|x|} \rangle d \sigma
\end{equation}
as $\eta \to 0$ and $\delta \to 0$. Letting $\eta \to 0$ and $\delta \to 0$ in \eqref{1753}, by \eqref{150111}-\eqref{15012} we finally get that
\begin{equation} \label{1756}
\int_{\Omega_t \cap B_r(0)} e^U =\int_{\partial \Omega_t \cap B_r(0)} |\nabla U|^{n-1} d \sigma-
\int_{\Omega_t \cap \partial B_r(0)}   |\nabla U|^{n-2} \langle \nabla U, \frac{x}{|x|} \rangle d \sigma.
\end{equation}
for all $r>0$ and a.e. $t\leq t_0$ (possibly depending on $r$) in view of the Lebesgue's monotone convergence theorem. 

\begin{oss} \label{rem0832} We aim to let $r \to +\infty$ in \eqref{1756}. In \cite{ChLi} no special care is required since for $n=2$  $U$ has a logarithmic behavior at infinity and then $\Omega_t$ is a bounded set. When $n>2$ we still don't know that $U$ behaves logarithmically at infinity and the validity of Theorem \ref{iso0944} is crucial in the next Section to establish such a property. Our argument relies instead on \eqref{1712} and on the finite measure property of $\Omega_t$, compare with \cite{Li}. 
\end{oss}
\noindent In radial coordinates we can write
\begin{equation} \label{1018}
|\Omega_t|= \int_0^\infty  dr \int_{\Omega_t \cap \partial B_r(0)} d\sigma <+\infty,\qquad
\int_{\mathbb{R}^n \setminus B_1(0)} \frac{|\nabla U|^q}{|x|^{2(n-q)}}=\int_1^\infty  \frac{dr}{r^{2(n-q)}} \int_{\partial B_r(0)} |\nabla U|^q d \sigma<+\infty
\end{equation}
in view of \eqref{1712}. We claim that for all $M\geq 1$ there exists $r\geq M$ so that
$$\int_{\Omega_t \cap \partial B_r(0)} d\sigma \leq \frac{1}{r} \quad \hbox{ and }\quad \frac{1}{r^{2(n-q)}} \int_{\partial B_r(0)} |\nabla U|^q d \sigma \leq \frac{1}{r}. $$
Indeed, if the claim were not true, we would find $M\geq 1$ so that for all $r \geq M$ there holds either
\begin{equation} \label{1049}
\int_{\Omega_t \cap \partial B_r(0)} d\sigma > \frac{1}{r} 
\end{equation}
or
\begin{equation} \label{1050}
\frac{1}{r^{2(n-q)}} \int_{\partial B_r(0)} |\nabla U|^q d \sigma > \frac{1}{r}. 
\end{equation}
Setting $I=\{r \geq M: \, \eqref{1049} \hbox{ holds}\}$ and $II=[M,\infty)\setminus I$, we have that
\begin{equation} \label{1135}
\int_I \frac{dr}{r}< \int_M^\infty  dr \int_{\Omega_t \cap \partial B_r(0)} d\sigma \leq |\Omega_t|
\end{equation}
and
\begin{equation} \label{11366}
\int_{II} \frac{dr}{r}< \int_M^\infty  \frac{dr}{r^{2(n-q)}} \int_{\partial B_r(0)} |\nabla U|^q d \sigma
\leq \int_{\mathbb{R}^n \setminus B_1(0)} \frac{|\nabla U|^q}{|x|^{2(n-q)}}
\end{equation}
since \eqref{1050} does hold for all $r \in II$. Summing up \eqref{1135}-\eqref{11366} we get that
$$\infty=\int_M^\infty \frac{dr}{r} \leq |\Omega_t|+\int_{\mathbb{R}^n \setminus B_1(0)} \frac{|\nabla U|^q}{|x|^{2(n-q)}}$$
in contradiction with \eqref{1018}, and the claim is established.

\medskip \noindent Thanks to the claim we can construct a sequence $r_k \to +\infty$ so that
\begin{equation} \label{1044}
\int_{\Omega_t \cap \partial B_{r_k}(0)} d\sigma \leq \frac{1}{r_k},\quad \frac{1}{r_k^{2(n-q)}} \int_{\partial B_{r_k}(0)} |\nabla U|^q d \sigma \leq \frac{1}{r_k}. 
\end{equation}
By \eqref{1044} and the H\"older's inequality we deduce the crucial estimate
\begin{equation} \label{10444}
\int_{\Omega_t \cap \partial B_{r_k}(0)} |\nabla U|^{n-1} d \sigma
\leq \left(\int_{\Omega_t \cap \partial B_{r_k}(0)} |\nabla U|^q d \sigma \right)^{\frac{n-1}{q}}
\left(\int_{\Omega_t \cap \partial B_{r_k}(0)} d \sigma \right)^{\frac{q-(n-1)}{q}}\leq
\frac{1}{r_k^{1-2\frac{(n-q)(n-1)}{q}}} \to 0 
\end{equation}
by choosing $q \in (n-1,n)$ sufficiently close to $n$.

\medskip \noindent Choosing $r=r_k$ in \eqref{1756} and letting $k \to +\infty$ we get that
\begin{equation} \label{1829}
\int_{\Omega_t} e^U =\int_{\partial \Omega_t } |\nabla U|^{n-1} d \sigma
\end{equation}
for a.e. $t\leq t_0$ in view of \eqref{10444}. Arguing as previously, by the co-area formula and the Lebesgue's differentiation theorem we have that
$$|\Omega_t|=\lim_{r \to +\infty} |\Omega_t \cap B_r(0)|=\lim_{r \to +\infty} \int_t^\infty ds \int_{\partial \Omega_s \cap B_r(0)} \frac{d\sigma}{|\nabla U|}
=\int_t^\infty ds \int_{\partial \Omega_s } \frac{d\sigma}{|\nabla U|},$$
and then
\begin{equation} \label{15011}
-\frac{d}{dt} |\Omega_t|=
\int_{\partial \Omega_t} \frac{d \sigma}{|\nabla U|}
\end{equation}
for a.e. $t\leq t_0$. Thanks to \eqref{1829}-\eqref{15011}, by the H\"older's and the isoperimetric inequalities we can now compute 
\begin{eqnarray}
-\frac{d}{ d  t} \left(\int_{\Omega_t } e^U dx\right)^{\frac{n}{n-1}}
&=& -\frac{n}{n-1} \left( \int_{\Omega_t} e^U dx \right)^{\frac{1}{n-1}} e^t \frac{d}{dt}|\Omega_t| \nonumber \\
&=&
\frac{n}{n-1}\left( \int_{\partial \Omega_t } |\nabla U|^{n-1} d\sigma \right)^{\frac{1}{n-1}} e^t \int_{\partial \Omega_t} \frac{d\sigma}{|\nabla U|} \nonumber \\
&\geq&
\frac{n}{n-1}  e^t |\partial \Omega_t|^{\frac{n}{n-1}}\geq (c_n \omega_n)^{\frac{1}{n-1}} e^t   |\Omega_t| \label{1736}
\end{eqnarray}
for a.e. $t \leq t_0$. Since $t \to \int_{\Omega_t } e^U dx $ is a monotone decreasing function, we get that
\begin{equation} \label{1159}
\left(\int_{\mathbb{R}^n} e^U dx\right)^{\frac{n}{n-1}} \geq \int_{-\infty}^{t_0} -\frac{d}{ d  t} \left(\int_{\Omega_t } e^U dx\right)^{\frac{n}{n-1}}
 dt \geq (c_n \omega_n)^{\frac{1}{n-1}} \int_{\mathbb{R}^n} e^U dx.
\end{equation}
Since by assumption $ \int_{\mathbb{R}^n} e^U dx \leq c_n \omega_n$, we get that
$$ \int_{\mathbb{R}^n} e^U dx = c_n \omega_n$$
and the inequalities in \eqref{1736}-\eqref{1159} are actually equalities. We have that for a.e. $t \leq t_0$
\begin{itemize}
\item $\Omega_t=B_{R(t)}(x(t))$ for some $R(t)>0$ and $x(t) \in \mathbb{R}^n$, since $\Omega_t$ in an extremal of the isoperimetric inequality
\item $|\nabla U|^{n-1}$ is a multiple of $\frac{1}{|\nabla U|}$ on $\partial \Omega_t$,
\item the function $M(t)=\int_{\Omega_t} e^U dx$ is absolutely continuous in $(-\infty,t_0)$  with
\begin{equation}\label{1907}
\frac{1}{n-1}M^{\frac{1}{n-1}}(t) M'(t)=\frac{1}{n} \frac{d}{ d  t} M^{\frac{n}{n-1}}(t)
=- (c_n \omega_n)^{\frac{1}{n-1}}\frac{\omega_n}{n} e^t   R^n(t).
\end{equation}
\end{itemize}

\medskip \noindent The aim now is to derive an equation for $M(t)$ by means of some Pohozaev identity. Let us emphasize that $U \in C^{1,\alpha}(\mathbb{R}^n)$ and the classical Pohozaev identities usually require more regularity. In \cite{DFSV} a self-contained proof  is provided in the quasilinear case, which reads in our case as
\begin{lm} \label{Po}
Let $\Omega \subset \mathbb{R}^{n}$, $n\geq2$, be a smooth bounded domain and $f$ be a locally Lipschitz continuous function. Then, there holds
$$n \int_{\Omega}F(U)=\int_{\partial\Omega} \left[F(U) \langle x-y,\nu\rangle
+|\nabla U|^{n-2}\langle x-y,\nabla U \rangle \partial_{\nu}U-\frac{|\nabla U|^n}{n}\langle x-y,\nu\rangle\right] $$
for all $y \in \mathbb{R}^n$ and all weak solution $U \in C^{1,\alpha}(\Omega)$ of $-\Delta_n U=f(U)$ in $\Omega$, where $F(t)=\displaystyle\int_{0}^{t}f(s)ds$ and $\nu$ is the unit outward normal vector at $\partial \Omega$.
\end{lm}
\noindent Let us re-write \eqref{1829} as
\begin{equation} \label{1906}
M(t)= n \omega_n |\nabla U|^{n-1} R^{n-1}(t)
\end{equation}
and use Lemma \ref{Po} on $\Omega_t=B_{R(t)}(x(t))$ with $y=x(t)$ to deduce
\begin{equation} \label{1908}
M(t)=\omega_n e^t R^n(t)  +\frac{n-1}{n} \omega_n  |\nabla U|^n R^n(t) 
\end{equation}
in view of $U=t$ and $|\nabla U|=-\partial_\nu U$ constant on $\partial \Omega_t$. By \eqref{1906}-\eqref{1908} we have that
\begin{equation} \label{1947}
\omega_n e^t R^n(t)=M(t)- (c_n \omega_n)^{-\frac{1}{n-1}} M^{\frac{n}{n-1}}(t),
\end{equation}
which, inserted into \eqref{1907}, gives rise to
\begin{equation}\label{1925}
M'(t)=- \frac{n-1}{n}(c_n \omega_n)^{\frac{1}{n-1}} M^{\frac{n-2}{n-1}}(t)+ \frac{n-1}{n} M (t)
\end{equation}
for a.e. $t\leq t_0$. Since $M$ is absolutely continuous in $\mathbb{R}$ and
$$\frac{1}{n-1}   \int  \frac{dM}{M-(c_n \omega_n)^{\frac{1}{n-1}} M^{\frac{n-2}{n-1}}}= \ln |M^{\frac{1}{n-1}} -(c_n \omega_n)^{\frac{1}{n-1}}  |,$$
we can integrate \eqref{1925} to get
\begin{equation} \label{1946}
M(t) =c_n \omega_n  \left[1-e^{\frac{t-t_0}{n}} \right]^{n-1}
\end{equation}
for all $t \leq  t_0$, in view of $M(t_0)=0$. Inserting \eqref{1946} into \eqref{1947} we deduce that
\begin{equation} \label{1211}
R^n(t)=c_n   \left[1-e^{\frac{t-t_0}{n}} \right]^{n-1} e^{-\frac{(n-1)t}{n}-\frac{t_0}{n}}
\end{equation}
for a.e. $t \leq t_0$. Since $R(t)$ is monotone, notice that \eqref{1211} is valid for all $t\leq t_0$ and can be re-written as
\begin{equation} \label{1216}
e^t=\frac{c_n \lambda^n}{(1+\lambda^{\frac{n}{n-1}}R^{\frac{n}{n-1}}(t))^n}
\end{equation}
where $\lambda=(\frac{e^{t_0}}{c_n})^{\frac{1}{n}}.$  To conclude, we just need to show that $x(t)=x_0$. First notice that a.e. $t_1,t_2\leq t_0$ either $x(t_1)=x(t_2)$ or, assuming for example $t_2<t_1$, $B_{R(t_1)}(x(t_1))\subset \subset B_{R(t_2)}(x(t_2))$ and $x(t_2)-R(t_2) \frac{x(t_2)-x(t_1)}{|x(t_2)-x(t_1)|} \in \partial B_{R(t_2)}(x(t_2))$ implies
$$R(t_2)-|x(t_1)-x(t_2)|=\big||x(t_2)-x(t_1)|-R(t_2)\big|=| x(t_2)-R(t_2) \frac{x(t_2)-x(t_1)}{|x(t_2)-x(t_1)|}-x(t_1)|> R(t_1).$$
In both cases, we have that $|x(t_2)-x(t_1)|\leq |R(t_2)-R(t_1)|$ for a.e. $t_1,t_2 \leq t_0$. Since $R \in C(-\infty,t_0]\cap C^1(-\infty,t_0)$, $x(t)$ can be uniquely extended as a map $\tilde x(t)$ which is continuous in $(-\infty,t_0]$ and locally Lipschitz in $(-\infty,t_0)$. Given $t<t_0$ we can alway find $t_n \downarrow t$ so that $\Omega_{t_n}=B_{R(t_n)}(x(t_n))$, $x(t_n)=\tilde x(t_n)$, and then there holds
$$\Omega_t=\bigcup_{n \in \mathbb{N}} \Omega_{t_n}=\bigcup_{n \in \mathbb{N}} B_{R(t_n)}(x(t_n))=B_{R(t)}(\tilde x(t))$$
by the continuity of $R(t)$ and $\tilde x(t)$. Identifying $x$ and $\tilde x$, we can assume that $x \in C(-\infty,t_0] \cap Lip_{\hbox{loc}}(-\infty,t_0)$ and $\Omega_t=B_{R(t)}(x(t))$ for all $t\leq t_0$. Use now the property $t=U(x(t)+R(t) \omega)$, $\omega \in \mathbb{S}^n$, to deduce
\begin{eqnarray*}
h&=& U(x(t+h)+R(t+h)\omega)-U(x(t)+R(t)\omega)= \langle \nabla U(x(t)+R(t)\omega), x(t+h)-x(t) \rangle\\
&&+
[R(t+h)-R(t)]\langle \nabla U(x(t)+R(t)\omega), \omega \rangle+o(|x(t+h)-x(t)|+|R(t+h)-R(t)|)
\end{eqnarray*}
as $h \to 0$, uniformly in $\omega \in \mathbb{S}^n$. Since $|\nabla U|$ is a non-zero constant on $\partial \Omega_t$ for a.e. $t\leq t_0$ and $\Omega_t=B_{R(t)}(x(t))$, we have that 
$$\nabla U(x(t)+R(t)\omega)=- |\nabla U| \omega,$$
and then, applied to $-\omega$ and $\omega$, it yields that
\begin{eqnarray*}
&& h=  |\nabla U| \langle x(t+h)-x(t),\omega \rangle- [R(t+h)-R(t)] |\nabla U| +o(|x(t+h)-x(t)|+|R(t+h)-R(t)|)\\
&& h=- |\nabla U| \langle x(t+h)-x(t),\omega \rangle- [R(t+h)-R(t)] |\nabla U| +o(|x(t+h)-x(t)|+|R(t+h)-R(t)|).
\end{eqnarray*}
Since $|\nabla U| \not= 0$, the difference then gives
$$\langle x(t+h)-x(t),\omega \rangle=o(|x(t+h)-x(t)|+|R(t+h)-R(t)|)$$
as $h \to 0$, uniformly in $\omega \in \mathbb{S}^n$. If $x(t+h)\not= x(t)$, the choice $\omega=\frac{x(t+h)-x(t)}{|x(t+h)-x(t)|}$ leads to
$$|\frac{x(t+h)-x(t)}{h}|\leq o(|\frac{R(t+h)-R(t)}{h}|) \to 0$$
as $h \to 0$. So we have shown that $x'(t)=\displaystyle \lim_{h\to 0}\frac{x(t+h)-x(t)}{h}=0$ for a.e. $t\leq t_0$.
Since $x \in Lip_{\hbox{loc}}(-\infty,t_0)$, by integration we deduce that $x(t)$ is constant for all $t \leq t_0$, say $x(t)=x_0$.

\medskip \noindent Given $x \in \mathbb{R}^n \setminus \{x_0\}$, by \eqref{1211} we can find a unique $t<t_0$ so that $R(t)=|x-x_0|$ and then
$$e^{U(x)}=\frac{c_n \lambda^n}{(1+\lambda^{\frac{n}{n-1}}|x-x_0|^{\frac{n}{n-1}})^n}$$
in view of \eqref{1216} and $U=t$ on $\partial B_{R(t)}(x_0)$. The proof is complete since we have shown that $U=U_{\lambda,x_0}$ for some $\lambda>0$ and $x_0 \in \mathbb{R}^n$.

\end{proof}

\section{Behavior of $U$ at infinity}
\noindent The estimates in Proposition \ref{T4} are not sufficient to establish the logarithmic behavior of $U$ at infinity but are essentially optimal in the limiting case $f \in L^1(\Omega)$. According to \cite{Ser1,Ser2}, a bit more regularity on $f$ gives $L^\infty$-bounds as stated in
\begin{pro} \label{Serrinest}
Let $f \in L^p(\Omega)$, $p>1$, and assume \eqref{a1}-\eqref{a2}. Let $u \in W^{1,n}_0(\Omega)$  be a weak solution of $-\hbox{div } {\bf a}(x,\nabla u)=f$. 
Then 
$$\|u \|_\infty \leq C (\frac{\|f\|_p}{d}+1)^{\alpha_0} (|\Omega|+1)^{\beta_0}  \|u\|^{\bar q}_{\frac{npq_1}{p-1}}$$
for some constants $C,\alpha_0,\beta_0,\bar q>0$ just depending on $n,p$ and $q_1\geq 1$.
\end{pro}
\begin{proof}
Given $q \geq 1$ and $k>0$ 
set 
$$F(s)=\left \{ \begin{array}{ll} s^q &\hbox{if } 0 \leq s \leq k\\  q k^{q-1} s-(q-1)k^q &\hbox{if }s\geq k  \end{array}\right.$$
and $G(s)=F(s) [F'(s)]^{n-1}$. Notice that $G$ is a piecewise $C^1-$function with a corner just at $s=k$ so that 
\begin{equation} \label{1508}
 [F'(s)]^n \leq G'(s) ,\qquad G(s) \leq q^{n-1} F^{\frac{n(q-1)+1}{q}}(s).
\end{equation}
Since $G(|u|)\in W^{1,n}_0(\Omega)$ for $G$ is linear at infinity, use $\hbox{sign}(u) G(|u|)$ as a test function in the equation of $u$ to get
\begin{eqnarray} \label{1607}
\int_\Omega |\nabla F(|u|)|^n \leq \frac{1}{d} \int_\Omega G'(|u|) \langle {\bf a}(x,\nabla u), \nabla  u \rangle=\frac{1}{d}
 \int_\Omega f \hbox{ sign}(u) G(|u|)
\end{eqnarray}
in view of \eqref{a2} and \eqref{1508}. Setting $m=\frac{p}{p-1}$ in view of $p>1$, by \eqref{1508} and the H\"older's inequality we deduce that
\begin{eqnarray} 
|\int_\Omega f \hbox{ sign}(u) G(|u|)| \leq  
q^{n-1} \int_\Omega |f| F^{\frac{n(q-1)+1}{q}}(|u|)
\leq q^{n-1} |\Omega|^{\frac{n-1}{mnq}} \|f\|_p \left(\int_\Omega F^{m n}(|u|) \right)^{\frac{n(q-1)+1}{mnq}}.
\label{1025}
\end{eqnarray}
The Sobolev embedding Theorem applied on $F(|u|)\in W^{1,n}_0(\Omega)$ now implies that
\begin{eqnarray*}
\left(\int_\Omega  F^{2mn}(|u|) \right)^{\frac{1}{2m}}
\leq C \int_\Omega |\nabla F(|u|)|^n \leq
\frac{C}{d} q^{n-1} |\Omega|^{\frac{n-1}{mnq}} \|f\|_p \left(\int_\Omega F^{m n}(|u|) \right)^{\frac{n(q-1)+1}{mnq}}
\end{eqnarray*}
for some $C\geq 1$ in view of \eqref{1607}-\eqref{1025}. Since $F(s) \to s^q$ in a  monotone way as $k \to +\infty$, we have that
\begin{eqnarray} \label{1622}
\left(\int_\Omega  |u|^{2m nq} \right)^{\frac{1}{2mq}}
\leq \hbox{exp }\left[ \frac{1}{q} \ln \frac{C\|f\|_p}{d} +\frac{(n-1) \ln  |\Omega|}{mnq^2} +(n-1) \frac{\ln q}{q} \right] \left(\int_\Omega |u|^{mnq} \right)^{\frac{1}{mq}[1-\frac{n-1}{nq}]}.
\end{eqnarray}
Assume now that $u \in L^{mnq_1}(\Omega)$ for some $q_1 \geq 1$. Setting $q_j=2^{j-1} q_1$, $j \in \mathbb{N}$, by iterating \eqref{1622} we deduce that
\begin{eqnarray*} 
&&\left(\int_\Omega  |u|^{m n q_{j+1}} \right)^{\frac{1}{m q_{j+1}}}
\leq 
\hbox{exp }\left[\frac{1}{q_j} \ln \frac{C\|f\|_p}{d} +\frac{(n-1) \ln  |\Omega|}{mnq_j^2}+(n-1) \frac{\ln q_j}{q_j}\right] \left[ \left(\int_\Omega |u|^{mnq_j} \right)^{\frac{1}{mq_j}} \right]^{1-\frac{n-1}{n q_j}}\\
&&
\leq 
\hbox{exp }\left[ \ln \frac{C\|f\|_p}{d} \sum_{k=j-1}^j \frac{a_k^j}{q_k}+\frac{(n-1) \ln  |\Omega|}{mn}\sum_{k=j-1}^j  \frac{a_k^j}{q_k^2} +(n-1) \sum_{k=j-1}^j \frac{a_k^j \ln q_k}{q_k} \right] \left[ \left(\int_\Omega |u|^{mnq_{j-1}} \right)^{\frac{1}{mq_{j-1}}} \right]^{a_{j-2}^j}\\
&&\dots \leq 
\hbox{exp }\left[\ln \frac{C\|f\|_p}{d} \sum_{k=1}^j \frac{a_k^j}{q_k}+\frac{(n-1) \ln  |\Omega|}{mn}\sum_{k=1}^j  \frac{a_k^j}{q_k^2}+(n-1) \sum_{k=1}^j \frac{a_k^j \ln q_k}{q_k} \right] 
\left(\int_\Omega |u|^{mnq_1} \right)^{\frac{a_0^j}{mq_1}}
\end{eqnarray*}
where 
$$a_k^j=\left\{\begin{array}{ll} [1-\frac{n-1}{n q_{k+1}}] \times \dots \times [1-\frac{n-1}{n q_j}]& \hbox{if }0\leq k <j\\
1& \hbox{if }k=j. \end{array} \right. $$
Since $a_k^j \leq 1$ for all $k=0,\dots,j$, we have that
\begin{eqnarray*}
\alpha_0&=& \frac{1}{n}\sup_{j \in \mathbb{N}} \sum_{k=1}^j \frac{a_k^j}{q_k}\leq
 \frac{1}{n}\sup_{j \in \mathbb{N}} \sum_{k=1}^j \frac{1}{q_k}
= \frac{2}{n} \sum_{k=1}^\infty  \frac{1}{q_1 2^k}<\infty\\
\beta_0 &=& \frac{n-1}{mn^2}\sup_{j \in \mathbb{N}} \sum_{k=1}^j  \frac{a_k^j}{q_k^2}\leq
\frac{4(n-1)}{mn^2} \sum_{k=1}^\infty  \frac{1}{q_1^2 4^k}<+\infty \\
\gamma_0&=&\frac{n-1}{n} \sup_{j \in \mathbb{N}} \sum_{k=1}^j \frac{a_k^j \ln q_k}{q_k}
\leq 2\frac{n-1}{n}  \sum_{k=1}^\infty \frac{(k-1)\ln 2+\ln q_1}{ q_1 2^k}<+\infty,
\end{eqnarray*}
and then it follows that
\begin{eqnarray} \label{1136} 
\left(\int_\Omega  |u|^{m n q_{j+1}} \right)^{\frac{1}{m n q_{j+1}}}
\leq 
\hbox{exp }\left[\alpha_0 \ln C(\frac{\|f\|_p}{d}+1)+ \beta_0 \ln ( |\Omega|+1)+ \gamma_0\right] \left(\int_\Omega |u|^{mnq_1} \right)^{\frac{a_0^j}{mn q_1}}.
\end{eqnarray}
Since
$$\bar q=\lim_{j \to +\infty} a_0^j= \prod_{k=1}^\infty (1-\frac{n-1}{nq_k})<\infty,$$
letting $j \to +\infty$ in \eqref{1136} we finally deduce that
\begin{eqnarray*}  
\|u \|_\infty \leq e^{\alpha_0 \ln C+\gamma_0} (\frac{\|f\|_p}{d}+1)^{\alpha_0} (|\Omega|+1)^{\beta_0}  \|u\|^{\bar q}_{mnq_1}
\end{eqnarray*}
and the proof is complete.
\end{proof}

\noindent Thanks to Theorem \ref{iso0944} we are just concerned with the range
\begin{equation} \label{1710}
\int_{\mathbb{R}^n} e^U \geq c_n\omega_n.
\end{equation}
By Proposition \ref{Serrinest} we can improve the estimates in Section 2 to get
\begin{thm} \label{aiutobis}
Let $U$ be a solution of \eqref{E1} which satisfies \eqref{1710}. Then $\hat U(x)=U(\frac{x}{|x|^2})$ satisfies
\begin{equation} \label{gamma0}
\hat U(x)-(\frac{\gamma_0}{n\omega_n})^{\frac{1}{n-1}}  \ln |x| \in L^\infty_{\hbox{loc}}(\mathbb{R}^n)
\end{equation}
and
\begin{equation} \label{gamma1}
\sup_{|x|=r} |x| \big|\nabla \left( \hat U(x)-(\frac{\gamma_0}{n\omega_n})^{\frac{1}{n-1}} \ln |x| \right) \big| \to 0
\end{equation}
for a sequence $r \to 0$, where $\gamma_0=\int_{\mathbb{R}^n} e^U$.
\end{thm}
\begin{proof} We adopt the same notations as in Theorem \ref{aiuto}, and we try to push more the analysis thanks to \eqref{1710}. Given $r>0$, recall that $\hat U$ has been decomposed in $B_r(0)$ as $\hat U=U_0+H_0$, $U_0,H_0 \in C^1_{\hbox{loc}} (\overline{B_r(0)}\setminus \{0\})$, where $H_0$ is a $n-$harmonic function in $B_r(0) \setminus \{0\}$ with $\displaystyle \sup_{B_r(0) \setminus \{0\}}H_0<+\infty$ and $U_0 \geq 0$ satisfies \eqref{pro1} with
$$U_0=0 ,\qquad \partial_\nu U_0 \leq 0 \hbox{ on }\partial B_r(0).$$

\medskip \noindent The desciption of the behavior of $H_0$ at $0$, as established in \cite{Ser1,Ser2}, has been later improved in \cite{KiVe} to show that there exists $\gamma \geq 0$ with
\begin{equation} \label{sing1}
H_0(x)-(\frac{\gamma}{n\omega_n})^{\frac{1}{n-1}} \ln |x|\in L^\infty (B_r(0)), \qquad \Delta_n H_0=\gamma \delta_0 \hbox{ in }\mathcal{D}'(B_r(0)).
\end{equation}
Since $\hat U \in W^{1,n-1}(B_r(0))$ according to Theorem \ref{aiuto}, we can extend \eqref{E1equiv} at $0$ as
\begin{equation}\label{1732}
-\Delta_n \hat U=\frac{e^{\hat U}}{|x|^{2n}}-\gamma_0 \delta_0
\end{equation}
in the sense
\begin{equation} \label{1734}
\int_{\mathbb{R}^n} |\nabla \hat U|^{n-2} \langle \nabla \hat U,\nabla \Phi \rangle=\int_{\mathbb{R}^n} \frac{e^{\hat U}}{|x|^{2n}} \Phi -\gamma_0 \Phi(0) 
\end{equation} 
for all $\Phi \in C^1(\mathbb{R}^n)$ so that $\hat \Phi \in W^{1,n}_{\hbox{loc}}(\mathbb{R}^n)$. Indeed, let us consider a smooth function $\eta$ so that $\eta=0$ for $|x|\leq \delta$, $\eta=1$ for $|x|\geq 2 \delta$ and $|\nabla \eta|\leq \frac{2}{\delta}$. Use $\eta[\Phi-\Phi(0)] \in \hat H$ as a test function in \eqref{meanE1equiv} to provide
\begin{equation} \label{1805}
\int_{\mathbb{R}^n}\eta |\nabla \hat U|^{n-2}\langle \nabla \hat U,\nabla \Phi \rangle+
O(\int_{\mathbb{R}^n} |\nabla \hat U|^{n-1}|\nabla \eta||\Phi-\Phi(0)|)=
\int_{\mathbb{R}^n} \eta \frac{e^{\hat U}}{|x|^{2n}}(\Phi-\Phi(0)).
\end{equation}
Since
$$\int_{\mathbb{R}^n} |\nabla \hat U|^{n-1}|\nabla \eta||\Phi-\Phi(0)|\leq
C \int_{B_{2\delta}(0)} |\nabla \hat U|^{n-1}\to 0$$
as $\delta \to 0$, we can let $\delta \to 0$ in \eqref{1805} and get the validity of \eqref{1734} in view of $ \gamma_0=\int_{\mathbb{R}^n} \frac{e^{\hat U}}{|x|^{2n}}=\int_{\mathbb{R}^n} e^U$

\medskip \noindent Since $U_0 \geq 0$, the singularity of $\hat U=U_0+H_0$ at $0$ should be weaker than that of $H_0$. Via an approximation procedure, it is easily seen that equations \eqref{sing1}-\eqref{1732} can be re-written as
\begin{eqnarray} 
&& \gamma \Phi(0)=\int_{\partial B_r(0)} |\nabla H_0|^{n-2} \partial_\nu H_0 \Phi-\int_{B_r(0)}  |\nabla H_0|^{n-2}\langle \nabla H_0,\nabla \Phi \rangle \label{1830} \\
&&\gamma_0 \Phi(0)=
\int_{B_r(0)} \frac{e^{\hat U}}{|x|^{2n}} \Phi +
\int_{\partial B_r(0)} |\nabla \hat U|^{n-2} \partial_\nu \hat U \Phi-
\int_{B_r(0)}  |\nabla \hat U|^{n-2}\langle \nabla \hat U,\nabla \Phi \rangle \label{1831}
\end{eqnarray} 
for all $\Phi \in C^1(B_r(0))$. We claim that
\begin{equation} \label{1840b}
|\nabla H_0|^{n-2}\partial_\nu H_0 \geq |\nabla \hat U|^{n-2}\partial_\nu \hat U \qquad \hbox{on }\partial B_r(0)
\end{equation} 
and then, by taking $\Phi=1$ in \eqref{1830}-\eqref{1831}, we deduce that
\begin{equation} \label{1842}
\gamma=\int_{\partial B_r(0)} |\nabla H_0|^{n-2} \partial_\nu H_0
\geq \int_{\partial B_r(0)} |\nabla \hat U|^{n-2} \partial_\nu \hat U 
=\gamma_0 -\int_{B_r(0)} \frac{e^{\hat U}}{|x|^{2n}}.
\end{equation} 
To establish the claim \eqref{1840b}, we write $H_0=\hat U-U_0$ and recall that $\nabla U_0=(\partial_\nu U_0)\nu$ with $\partial_\nu U_0 \leq 0$ on $\partial B_r(0)$. 
Since
$$ |\nabla H_0|^{n-2}=\left[ |\nabla \hat U|^2+ (\partial_\nu U_0)^2-2\partial_\nu \hat U \partial_\nu U_0 \right]^{\frac{n-2}{2}},$$
when $\partial_\nu \hat U\geq 0$ we have that
$$|\nabla H_0|^{n-2} \geq  |\nabla \hat U|^{n-2},\quad \partial_\nu H_0 \geq \partial_\nu \hat U\geq 0$$
and then \eqref{1840b} does hold. When $ \partial_\nu U_0\leq \partial_\nu \hat U<0$ there holds $\partial_\nu H_0  \geq 0$ and then
$$|\nabla H_0|^{n-2} \partial_\nu H_0 \geq  0 > |\nabla \hat U|^{n-2} \partial_\nu \hat U.$$
When $ \partial_\nu \hat U<\partial_\nu U_0$ we have that 
\begin{eqnarray*}
|\nabla H_0|^{n-2} \leq  |\nabla \hat U|^{n-2},\quad 0> \partial_\nu H_0 \geq \partial_\nu \hat U
\end{eqnarray*}
and then \eqref{1840b} does hold.

\medskip \noindent Since $(\frac{\gamma_0}{n\omega_n})^{\frac{1}{n-1}}\geq \frac{n^2}{n-1}$ in view of \eqref{1710}, by \eqref{sing1} and \eqref{1842} we have that
\begin{equation} \label{1951}
\frac{e^{H_0}}{|x|^{2n}} \in L^q(B_r(0))
\end{equation}
for all $1\leq q<\frac{n-1}{n-2}$ if $r$ is sufficiently small. By \eqref{pro1} and \eqref{1951} it follows that
\begin{equation} \label{1952}
\frac{e^{\hat U}}{|x|^{2n}}=e^{U_0} \frac{e^{H_0}}{|x|^{2n}} \in L^q(B_r(0))
\end{equation}
for all $1\leq q<\frac{n-1}{n-2}$ if $r>0$ is sufficiently small. Thanks to \eqref{1952} we can apply Proposition \ref{Serrinest} to $U_\epsilon$ on $A_\epsilon$ (see \eqref{eqHepsilon}-\eqref{eqUepsilon}) with ${\bf a}(x,p)$ given by \eqref{2022} to get
$$\|U_\epsilon \|_{\infty, A_\epsilon} \leq C$$
for some uniform $C>0$.  We have used that 
$$\sup_{\epsilon} \|U_\epsilon\|_{p,A_\epsilon}<+\infty$$ 
for all $p\geq 1$ in view of  \eqref{17255} and the Sobolev embedding Theorem. Letting $\epsilon \to 0$ we get that $\|U_0\|_{\infty,B_r(0)}<+\infty$ and then
\begin{equation} \label{1610}
\hat U=U_0+H_0=(\frac{\gamma}{n\omega_n})^{\frac{1}{n-1}} \ln |x|+H(x), \quad H \in L^\infty_{\hbox{loc}}(\mathbb{R}^n)
\end{equation}
in view of \eqref{sing1}. Notice that now $\gamma$ does not depend on $r$ and then satisfies
$$\gamma \geq c_n \omega_n$$ 
in view of \eqref{1710} and \eqref{1842}. Given $r>0$ small, let us define the function 
$$V_r(y)=\hat U(ry)-(\frac{\gamma}{n\omega_n})^{\frac{1}{n-1}} \ln r=(\frac{\gamma}{n\omega_n})^{\frac{1}{n-1}} \ln |y|+H(ry).$$
Since
$$\Delta_n V_r=- \frac{e^{\hat U(ry)}}{r^n |y|^{2n}}=
- \frac{r^{\frac{n}{n-1}+\alpha}e^{H(ry)}}{|y|^{\frac{n(n-2)}{n-1} -\alpha}}$$
in view of \eqref{1610} with $\alpha= (\frac{\gamma}{n\omega_n})^{\frac{1}{n-1}}-\frac{n^2}{n-1} \geq 0$, we have that $V_r$ and $\Delta_n V_r$ are bounded in $L^\infty_{\hbox{loc}}(\mathbb{R}^n \setminus \{0\})$, uniformly in $r$. By \cite{Dib,Ser1,Tol} we deduce that $V_r$ is bounded in $C^{1.\alpha}_{\hbox{loc}}(\mathbb{R}^n \setminus \{0\})$, uniformly in $r$. By the Ascoli-Arzel\'a's Theorem and a diagonal process we can find a sequence $r \to 0$ so that $V_r \to V_0$ in $C^1_{\hbox{loc}}(\mathbb{R}^n \setminus \{0\})$, where $V_0$ is a n-harmonic function in $\mathbb{R}^n \setminus \{0\}$. Setting $H_r(y)=H(ry)$, we deduce that $H_r \to H_0$ in $C^1_{\hbox{loc}}(\mathbb{R}^n \setminus \{0\})$, where $H_0 \in L^\infty(\mathbb{R}^n)$ in view of \eqref{1610}.
Since $V_0=(\frac{\gamma}{n\omega_n})^{\frac{1}{n-1}} \ln |y|+H_0$ with $H_0 \in L^\infty (\mathbb{R}^n)\cap C^1(\mathbb{R}^n \setminus \{0\})$, we can apply Lemma \ref{1804} below to show that $H_0$ is a constant function. In particular we get that
\begin{equation} \label{1813}
\sup_{|x|=r} |x| \Big|\nabla \left( \hat U(x)-(\frac{\gamma}{n\omega_n})^{\frac{1}{n-1}} \ln |x| \right) \Big| =\sup_{|y|=1} |\nabla H_r(y)| \to \sup_{|y|=1} |\nabla H_0(y)| =0 \end{equation}
along the sequence $r \to 0$. The proof of  \eqref{gamma0}-\eqref{gamma1} now follows by \eqref{1610}-\eqref{1813} once we show that $\gamma=\gamma_0$. Indeed, by \eqref{1831} we have that
$$\gamma_0=
\int_{B_r(0)} \frac{e^{\hat U}}{|x|^{2n}} +
\int_{\partial B_r(0)} |\nabla \hat U|^{n-2} \partial_\nu \hat U =
o(1) + \frac{\gamma}{n\omega_n}  \int_{\partial B_r(0)} \frac{1}{|x|^{n-1}}(1+o(1)) \to \gamma$$
where $r \to 0$ is any sequence with property \eqref{1813}. The proof is complete.
\end{proof}
\noindent We have used the following simple result:
\begin{lm} \label{1804} Let $\gamma \ln |x|+H$ be a $n-$harmonic function in $\mathbb{R}^n \setminus \{0\}$ with $H \in C^1(\mathbb{R}^n \setminus \{0\})$. If $H \in L^\infty(\mathbb{R}^n)$, then $H$ is a constant function.
\end{lm}
\begin{proof}
Let $\eta$ be a cut-off function with compact support in $\mathbb{R}^n \setminus \{0\}$. Since
$$-\Delta_n (\gamma \ln |x|+H)=-\Delta_n(\gamma \ln |x|+H)+\Delta_n(\gamma \ln |x|)=0 \qquad \hbox{in }\mathbb{R}^n \setminus \{0\},$$
we can use $\eta^n H$ as a test function to get
\begin{eqnarray*}
d \int_{\mathbb{R}^n} \eta ^n |\nabla H|^n &\leq& \int_{\mathbb{R}^n} \eta^n \langle |\nabla (\gamma \ln |x|+H)|^{n-2} \nabla (\gamma \ln |x|+H)-|\nabla (\gamma \ln |x|)|^{n-2}\nabla (\gamma \ln |x|), \nabla H \rangle \\
&=& - n \int_{\mathbb{R}^n } \eta^{n-1}  H \langle |\nabla (\gamma \ln |x|+H)|^{n-2} \nabla (\gamma \ln |x|+H)-|\nabla (\gamma \ln |x|)|^{n-2}\nabla (\gamma \ln |x|), \nabla \eta \rangle
\end{eqnarray*}
in view of \eqref{1045}. Since $H \in L^\infty (\mathbb{R}^n)$, by the Young's inequality we get that
\begin{eqnarray*}
d \int_{\mathbb{R}^n} \eta ^n |\nabla H|^n \leq C n \|H\|_\infty  \int_{\mathbb{R}^n } \eta^{n-1}  \left[|\nabla H|^{n-1}+\frac{|\nabla H|}{|x|^{n-2}}\right]|\nabla \eta|
\leq \frac{d}{2}  \int_{\mathbb{R}^n} \eta ^n |\nabla H|^n+C \left[\int_{\mathbb{R}^n} |\nabla \eta|^n+
\int_{\mathbb{R}^n} \frac{|\nabla \eta|^{\frac{n}{n-1}}}{|x|^{\frac{n(n-2)}{n-1}}} \right]
\end{eqnarray*}
in view of $\eta \leq 1$ and
$$ ||v+w|^{n-2}(v+w)-|w|^{n-2}w|\leq C (|v|^{n-1}+|v||w|^{n-2}).$$
Hence, we have found that
\begin{eqnarray} \label{16077}
\int_{\mathbb{R}^n} \eta ^n |\nabla H|^n \leq C \left[\int_{\mathbb{R}^n} |\nabla \eta|^n+
\int_{\mathbb{R}^n} \frac{|\nabla \eta|^{\frac{n}{n-1}}}{|x|^{\frac{n(n-2)}{n-1}}} \right].
\end{eqnarray}
Given $\delta \in (0,1)$, we make the following choice for $\eta$:
$$\eta(x)=\left\{ \begin{array}{ll}0  &\hbox{if }|x|\leq \delta^2\\
-\frac{\ln |x|-2 \ln\delta}{\ln \delta}  & \hbox{if }\delta^2 \leq |x|\leq \delta\\
1& \hbox{if }\delta \leq |x|\leq \frac{1}{\delta}\\
\frac{\ln |x|+2\ln \delta}{\ln \delta} &\hbox{if }\frac{1}{\delta}\leq |x|\leq \frac{1}{\delta^2}\\
0& \hbox{if }|x|\geq \frac{1}{\delta^2}.
\end{array} \right.$$
Since
$$\int_{\mathbb{R}^n} |\nabla \eta|^n =
\frac{2}{|\ln \delta|^n}\int_{\{\delta^2 \leq |x| \leq \delta\}} \frac{1}{|x|^n}=
\frac{2 \omega_{n-1}}{|\ln \delta|^{n-1}}
\to 0$$
 and
$$  \int_{\mathbb{R}^n} \frac{|\nabla \eta|^{\frac{n}{n-1}}}{|x|^{\frac{n(n-2)}{n-1}}}=
\frac{2}{|\ln \delta|^{\frac{n}{n-1}}}\int_{\{\delta^2 \leq |x| \leq \delta\}} \frac{1}{|x|^n}
=\frac{2 \omega_{n-1}}{|\ln \delta|^{\frac{1}{n-1}}}\to 0$$
as $\delta \to 0$, we deduce that
$$ \int_{\mathbb{R}^n} |\nabla H|^n=0$$ 
by letting $\delta \to 0$ in \eqref{16077}. Then $H$ is a constant function.
\end{proof}

\section{Pohozaev identity}
\noindent Thanks to Theorem \ref{aiutobis}, we aim to apply the Pohozaev identity of Lemma \ref{Po} to show that \eqref{1710} automatically implies $\int_{\mathbb{R}^n} e^U =c_n\omega_n$. Combined with Theorem \ref{iso0944}, it completes the proof of the classification result in Theorem \ref{thm1}.

\medskip \noindent To this aim, we show the following:
\begin{thm}
Let $U$ be a solution of \eqref{E1} which satisfies \eqref{1710}. Then, there holds
$$\int_{\mathbb{R}^n} e^U = c_n \omega_n.$$
\end{thm}
\begin{proof} Since
$$\partial_i U(x)= \sum_{k=1}^n  \frac{1}{|x|^2} \left(\delta_{ik}-2\frac{x_i x_k}{|x|^2}\right) (\partial_k \hat U)(\frac{x}{|x|^2}),$$
we have that
$$|\nabla U|(x)=\frac{1}{|x|^2} |\nabla \hat U| (\frac{x}{|x|^2}),\qquad
\langle x,\nabla U(x)\rangle=-\langle \frac{x}{|x|^2},\nabla \hat U(\frac{x}{|x|^2})\rangle.$$
We can apply Theorem \ref{aiutobis} and deduce by \eqref{gamma1} that
\begin{equation} \label{1931}
 |\nabla U|(x)=\frac{1}{|x|} [(\frac{\gamma_0}{n\omega_n})^{\frac{1}{n-1}} + o(1)],\qquad
\langle x,\nabla U(x) \rangle=-(\frac{\gamma_0}{n\omega_n})^{\frac{1}{n-1}}+o(1)
\end{equation}
uniformly for $x \in \partial B_R(0)$, for a sequence $R =\frac{1}{r} \to +\infty$ and $\gamma_0=\int_{\mathbb{R}^n} e^U$. By \eqref{1931} we have that
\begin{equation} \label{2115}
\int_{\partial B_R(0)} \left[|\nabla U|^{n-2}\langle x,\nabla U \rangle \partial_{\nu}U-\frac{|\nabla U|^n}{n}\langle x,\nu\rangle\right]\to 
\omega_{n-1}(1-\frac{1}{n})(\frac{\gamma_0}{n\omega_n})^{\frac{n}{n-1}} 
\end{equation}
as $R\to +\infty$. Since by \eqref{gamma0} 
$$|x|^{(\frac{\gamma_0}{n\omega_n})^{\frac{1}{n-1}}}e^U \in L^\infty (\mathbb{R}^n \setminus B_1(0) )$$
with $(\frac{\gamma_0}{n\omega_n})^{\frac{1}{n-1}}\geq \frac{n^2}{n-1}$ in view of \eqref{1710}, we also get that
\begin{equation} \label{2123}
\int_{\partial B_R(0)} e^U  \langle x,\nu\rangle \to 0 \end{equation}
as $R\to +\infty$. We apply Lemma \ref{Po} to $U$ on $B_R(0)$ with $y=0$
%$f(t)=e^t$ 
and let $R\to +\infty$ to get
$$n \gamma_0=\omega_n(n-1)(\frac{\gamma_0}{n\omega_n})^{\frac{n}{n-1}} $$
in view of \eqref{2115}-\eqref{2123}. It results that
$$\gamma_0=\int_{\mathbb{R}^n} e^U=c_n \omega_n.$$
\end{proof}

\bibliographystyle{plain}

\end{document}